\documentclass[12pt]{article}
\usepackage{amsmath, amsthm, amssymb}
\usepackage{enumerate}
\usepackage{hyperref}
\usepackage{xspace}

\title{Generating random graphs in biased Maker-Breaker games}

\author{
Asaf Ferber
\thanks{Institute of Theoretical Computer Science, ETH Zurich, 8092 Zurich,
Switzerland. Email: \texttt{asaf.ferber@inf.ethz.ch}.}
\and
Michael Krivelevich
\thanks{School of Mathematical Sciences,
Raymond and Beverly Sackler Faculty of Exact Sciences,
Tel Aviv University, Tel Aviv, 69978, Israel.
Email: \texttt{krivelev@post.tau.ac.il}.
Research supported in part by USA-Israel BSF Grant 2010115 and by
grant 912/12 from the Israel Science Foundation.}
\and
Humberto Naves
\thanks{Department of Mathematics, ETH, 8092 Zurich, Switzerland
and Department of Mathematics, UCLA, Los Angeles, CA 90095 USA.
Email: \texttt{hnaves@math.ucla.edu}.}}

\newif\ifnotesw
\noteswtrue     

\newtheorem{theorem}{Theorem}[section]

\newtheorem{lemma}[theorem]{Lemma}
\newtheorem{claim}[theorem]{Claim}
\newtheorem{corollary}[theorem]{Corollary}

\newtheorem{definition}[theorem]{Definition}

\newtheorem{remark}[theorem]{Remark}

\newcommand{\secref}[1]{Section~\ref{#1}\xspace}
\newcommand{\thmref}[1]{Theorem~\ref{#1}\xspace}
\newcommand{\lemref}[1]{Lemma~\ref{#1}\xspace}

\newcommand{\corref}[1]{Corollary~\ref{#1}\xspace}
\newcommand{\defref}[1]{Definition~\ref{#1}\xspace}

\newcommand{\clref}[1]{Claim~\ref{#1}\xspace}

\newcommand{\Prob}{\mathbb{P}}
\newcommand{\gnp}{\ensuremath{\mathbb G}}
\newcommand{\dnp}{\ensuremath{\mathbb D}}
\newcommand{\Bin}{\textup{Bin}}
\newcommand{\eps}{\varepsilon}
\newcommand{\mc}[1]{\mathcal #1}

\newcommand{\aas}{a.a.s.\ }

\newcommand{\danger}{\mathsf {dang}}
\newcommand{\avdanan}{\overline{\danger}}

\begin{document}
\maketitle

\begin{abstract}

We present a general approach connecting biased Maker-Breaker games
and problems about local resilience in random graphs. We utilize
this approach to prove new results and also to derive some known
results about biased Maker-Breaker games. In particular, we show
that for $b=o\left(\sqrt{n}\right)$, Maker can build a pancyclic graph
(that is, a graph that contains cycles of every possible length) while
playing a $(1:b)$ game on $E(K_n)$. As another application, we show
that for $b=\Theta\left(n/\ln n\right)$, playing a $(1:b)$ game on
$E(K_n)$, Maker can build a graph which contains copies of all spanning
trees having maximum degree $\Delta=O(1)$ with a bare path of linear
length (a bare path in a tree $T$ is a path with all interior
vertices of degree exactly two in $T$).

\end{abstract}

\section{Introduction}

Let $X$ be a finite set and let $\mc F \subseteq 2^X$ be a family of
subsets. In the $(a:b)$ Maker-Breaker game $\mc F$, two players,
called Maker and Breaker, take turns in claiming previously
unclaimed elements of $X$. The set $X$ is called the \emph{board} of
the game and the members of $\mc F$ are referred to as the
\emph{winning sets}. Maker claims $a$ board elements per turn,
whereas Breaker claims $b$ elements. The parameters $a$ and $b$ are
called the \emph{bias} of Maker and of Breaker, respectively. Maker
wins the game as soon as he occupies all elements of some winning
set. If Maker does not fully occupy any winning set by the time
every board element is claimed by either of the players, then
Breaker wins the game. We say that the $(a:b)$ game $\mc F$ is
\emph{Maker's win} if Maker has a strategy that ensures his victory
against any strategy of Breaker, otherwise the game is
\emph{Breaker's win}. The most basic case is $a=b=1$, the so-called
\emph{unbiased} game, while for all other choices of $a$ and $b$ the
game is called a \emph{biased} game. Note that being the first
player is never a disadvantage in a Maker-Breaker game. Therefore,
in order to prove that Maker can win some Maker-Breaker game as the
first or the second player it is enough to prove that he can win
this game as a second player. In this paper we are concerned with
providing winning strategies for Maker and hence we will always
assume that Maker is the second player to move.

It is natural to play Maker-Breaker games on the edge set of a graph
$G=(V,E)$. In this case, $X=E$ and the winning sets are all the edge
sets of the edge-minimal subgraphs of $G$ which possess some given
monotone increasing graph property $\mc P$. We refer to this
game as the $(a:b)$ game $\mc P(G)$. In the \emph{connectivity
game} Maker wins if and only if his edges contain a spanning tree.
In the \emph{perfect matching} game the winning sets are all sets of
$\lfloor |V(G)|/2 \rfloor$ independent edges of $G$. Note that if
$|V(G)|$ is odd, then such a matching covers all vertices of $G$ but
one. In the \emph{Hamiltonicity game} the winning sets are all edge
sets of Hamilton cycles of $G$. Given a positive integer $k$, in the
\emph{$k$-connectivity game} the winning sets are all edge sets of
$k$-vertex-connected spanning subgraphs of $G$. Given a graph $H$,
in the \emph{$H$-game} played on $G$, the winning sets are all the
edge sets of copies of $H$ in $G$.

Playing unbiased Maker-Breaker games on the edge set of $K_n$ is
frequently in favor of Maker. For example, it is easy to see (and
also follows from \cite{Lehman}) that for every $n\ge 4$, Maker can
win the unbiased connectivity game in $n-1$ moves (which is clearly
also the fastest possible strategy). Other unbiased games played on
$E(K_n)$ like the perfect matching game, the Hamiltonicity game, the
$k$-vertex-connectivity game and the $T$-game where $T$ is a
spanning tree with bounded maximum degree, are also known to be easy
win for Maker (see e.g, \cite{CFGHL}, \cite{FH}, \cite{HKSS2009b}).
It is thus natural to give Breaker more power by allowing him to claim
$b>1$ elements in each turn.

Note that Maker-Breaker games are known to be \emph{bias monotone}.
That means that none of the players can be harmed by claiming more
elements. Therefore, it makes sense to study $(1:b)$ games and the
parameter $b^*$ which is the \emph{critical bias} of the game, that
is, $b^*$ is the maximal bias $b$ for which Maker wins the
corresponding $(1:b)$ game $\mc F$.

There is a striking relation between the theory of biased
Maker-Breaker games and the theory of random graphs, frequently
referred to as the \emph{Erd\H{o}s paradigm}. Roughly speaking, it
suggests that the critical bias for the game played by two ``clever
players'' and the appropriately defined critical bias for the game
played by two ``random players'' are asymptotically the same. In
this ``random players'' version of the game, both players use the
\emph{random strategy}, i.e., Maker claims one random unclaimed
element, while Breaker claims $b$ random unclaimed elements from the
board $E(K_n)$, per move. Note that the resulting graph occupied by
Maker at the end of the game is the random graph $\gnp(n,m)$, chosen
uniformly among all graphs with $n$ vertices and
$m=\lfloor\frac{1}{1+b}\binom{n}{2}\rfloor$ edges.
Therefore, if the winning sets consist of all the
edge sets of subgraphs of $K_n$ which possess
some monotone graph property $\mc P$, a natural guess for the
critical bias is $b^*$ for which $m^*=\frac{1}{1+b^*}\binom{n}{2}$ is the
threshold for the property that $\gnp(n,m)$ typically possesses
$\mc P$. For this reason, the Erd\H{os} paradigm is also known
as the \emph{random graph intuition}.

Chv\'atal and Erd\H{o}s were the first to indicate this phenomenon
in their seminal paper \cite{CE}. They showed that Breaker, playing
with bias $b=\frac{(1+\eps)n}{\ln n}$, can isolate a vertex
in Maker's graph while playing on the board $E(K_n)$. It thus
follows that Breaker wins every game for which the winning sets
consist of subgraphs of $K_n$ with positive minimum degree. What is
most surprising about their result is that at the end of the game,
Maker's graph consists of roughly $m=\frac 12 n\ln n$ edges which is
(asymptotically) the threshold for a random graph $\gnp(n,m)$ to stop
``having isolated vertices" (for more details on properties'
thresholds for random graphs, the reader is referred to
\cite{BolRand} and \cite{JLR}). In this spirit, the results of
Chv\'atal and Erd\H{o}s in \cite{CE} hint that $b^*=\frac{n}{\ln n}$
is actually the critical bias for many games whose target sets
consist of graphs having some property $\mc P$, for which the
threshold is $m^*=\frac{1}2 n\ln n$ (such as the
connectivity game, the perfect matching game and the Hamiltonicity
game). Gebauer and Szab\'o showed in \cite{GS} that the critical
bias for the connectivity game played on $E(K_n)$ is asymptotically
equal to $n/\ln n$. In a relevant development, Krivelevich proved in
\cite{Krive} that the critical bias for the Hamiltonicity game is
indeed $(1+o(1))n/\ln n$.

Another striking result exploring the relation between results in
Maker-Breaker games played on graphs and threshold probabilities for
properties of random graphs is due to Bednarska and \L uczak in
\cite{BL}. Given a graph $G$ on at least three vertices we define
\[
  m(G)=\max \left\{\frac{|E(H)|-1}{|V(H)|-2} : H\subseteq G
  \text{ and } |V(H)|\ge 3\right\}.
\]
Bednarska and \L uczak proved that the critical bias for the
$H$-game is of order $\Theta\left(n^{1/m(H)}\right)$.
The most surprising part in their proof is the side of Maker,
where they proved the following:

\begin{theorem}[Theorem 2 in \cite{BL}]\label{thm:BedLuc}
For every graph $H$ which contains a cycle there exists a constant
$c_0$ such that for every sufficiently large integer $n$ and $b\le
c_0n^{1/m(H)}$ Maker has a random strategy for the $(1:b)$ $H$-game
played on $E(K_n)$ that succeeds with probability $1-o(1)$ against
any strategy of Breaker.
\end{theorem}

Stating it intuitively, they proved that an ``optimal" strategy for
Maker is just to claim edges at random without caring about
Breaker's moves! Note that since a Maker-Breaker game is a
deterministic game, it follows that if Maker has a random strategy
that works with non-zero probability against any given strategy of
Breaker, then the game is Maker's win (otherwise Maker's strategy
should work with probability zero against Breaker's winning
strategy).

In the proof of \thmref{thm:BedLuc}, the graph obtained by
Maker at the end of the game is not exactly a random graph, since
some \emph{failure} edges might exist (that is, it might happen that
by choosing random edges, Maker attempts occasionally to pick an
edge $e$ which already belongs to Breaker). Thus, in order to prove
their result, Bednarska and \L uczak not only proved that random
graphs typically contain copies of the target graph $H$, but they
also showed that with a positive probability, even after removing a
small fraction of the total number of edges, these graphs still
contain many copies of $H$. This particular statement relates to the
\emph{resilience} of random graphs with respect to the property
``containing a copy of $H$''.

Given a monotone increasing graph property $\mc P$ and a graph
$G$ which satisfies $\mc P$, the \emph{resilience of $G$ with
respect to $\mc P$} measures how much one should change $G$ in
order to destroy $\mc P$ (here we assume that an edgeless graph
does not satisfy $\mc P$). There are two natural ways to define
it quantitatively. The first one is the following:
\begin{definition}\label{def:global}
For a monotone increasing graph property $\mc P$, the
\emph{global resilience} of $G$ with respect to $\mc P$ is the
minimum number $0\le r\le 1$ such that by deleting $r\cdot e(G)$
edges from $G$ one can obtain a graph $G'$ not having $\mc P$.

\end{definition}

Since one can destroy many natural properties by small changes (for
example, by isolating a vertex), it is natural to limit the number
of edges touching any vertex that one is allowed to delete. This
leads to the following definition of \emph{local resilience}.

\begin{definition} \label{def:local}
For a monotone increasing graph property $\mc P$, the
\emph{local resilience} of $G$ with respect to $\mc P$ is the
minimum number $0\le r\le 1$ such that by deleting at each vertex
$v$ at most $r\cdot d_G(v)$ edges one can obtain a graph not having
$\mc P$.
\end{definition}
Sudakov and Vu initiated the systematic study of resilience of
random and pseudorandom graphs in \cite{SV}. Since then, this field
has attracted substantial research interest (see, e.g.
\cite{BCS,BKS,BKS2,BKT,FK,KLS,LS}).

Going back to \thmref{thm:BedLuc}, Bednarska and \L uczak
actually proved that playing according to the random strategy, Maker
can typically build a graph $G\sim \gnp(n,m)$ minus some
$\eps$-fraction of its edges. They then showed that for a
given graph $H$ and an appropriate $m$, the global resilience of a
typical $G\sim \gnp(n,m)$ with respect to the property ``containing a
copy of $H$" is at least $\eps$. It is thus natural to seek
an alternative theorem which provides the analogous local resilience
argument.

The main result in this paper uses a sophisticated version of the
argument in \cite{BL}. Let $G$ be a graph and let $0<p<1$. The model
$\gnp(G,p)$ is a random subgraph $G'$ of $G$, obtained by
retaining each edge of $G$ in $G'$ independently at random with
probability $p$. For the special case where $G=K_n$, we denote
$\gnp(n,p):=\gnp(K_n,p)$, which is the well-known
Erd\H{o}s-R\'enyi model of random graphs.
Let $\mc P$ be a graph property, and consider sequences
of graphs $\{G_n\}_{n=1}^\infty$ (indexed by the number of vertices)
and probabilities $\{p(n)\}_{n=1}^\infty$. We say that $\gnp(G_n, p(n))
\in \mc P$ \emph{asymptotically almost surely},
or \aas for brevity, if the probability that $\gnp(G_n, p(n))\in\mc P$
tends to $1$ as $n$ goes to infinity. In this paper, we often abuse
notation and simply write $G=G_n$ and $p=p(n)$ to denote those
sequences. Before stating our main result we need the following
definition:
\begin{definition}
\label{PisRes}
Let $\mc P$ be a monotone increasing graph property, let
$G=G_n$ denote a family of graphs (where $G_n$ is a graph on $n$
vertices), let $0< p=p(n)< 1$ and let $0\le r\le 1$. We
say that $\mc P$ is $(G,p,r)$-\emph{resilient} if the local
resilience of a graph $G'\sim \gnp(G,p)$ with respect to
$\mc P$ is \aas at least $r$.
\end{definition}

Our main result is the following.

\begin{theorem} \label{thm:main}
For every constant $0<\eps\leq 1/100$ and a sufficiently large
integer $n$ the following holds. Suppose that
\begin{enumerate} [$(i)$]
\item $0< p=p(n)< 1$,
\item $G$ is a graph with $|V(G)|=n$,
\item $\delta(G)\ge \frac{10\ln n}{\eps p}$, and
\item $\mc P$ is a monotone increasing graph property which is
$(G,p,\eps)$-resilient.
\end{enumerate}
Then Maker has a winning strategy in the $(1:
\lfloor\frac{\eps}{20p}\rfloor)$ game $\mc P(G)$.
\end{theorem}

\thmref{thm:main} connects between Maker's side in biased
Maker-Breaker games on graphs and local resilience; it thus allows
to use (known) results about local resilience to give a lower
estimate for the critical bias in biased Maker-Breaker games. We now
present our concrete results for biased games, all of them are
applications of \thmref{thm:main} and corresponding local resilience
results for random graphs.

First, as a warm up we prove the following theorem which shows that
the critical bias for the Hamiltonicity game played on $E(K_n)$ is
$\Theta(\frac{n}{\ln n})$.

\begin{theorem}\label{thm:app1}
There exists a constant $\alpha>0$ for which for every sufficiently
large integer $n$ the following holds. Suppose that $b\le \alpha
n/\ln n$, then Maker has a winning strategy in the $(1:b)$
Hamiltonicity game played on $E(K_n)$.
\end{theorem}

This result is presented here mainly for illustrative purposes, and
also due to the historical importance of the biased Hamiltonicity
game and the long road it took before having been resolved finally
in \cite{Krive}.

As a second application we show that by playing a $(1:b)$ game on
$E(K_n)$, if $b=o(\sqrt{n})$, then Maker wins the \emph{pancyclicity}
game. That is, Maker can build a graph which consists of cycles of any
given length $3\le \ell \le n$.

\begin{theorem} \label{thm:app2}
Let $b=o(\sqrt{n})$. Then in the $(1:b)$ game played on $E(K_n)$, Maker
has a winning strategy in the pancyclicity game.
\end{theorem}

Note that the result is asymptotical tight possible in the sense that for
$b\ge 2\sqrt{n}$, Chv\'atal and Erd\H{o}s showed in \cite{CE} that Maker
cannot even build a triangle.
%

Ferber, Hefetz, and Krivelevich showed in \cite{FHK} that if $T$ is
a tree on $n$ vertices and $\Delta(T)\le n^{0.05}$ then the
following holds. In the $(1:b)$ game, Maker has a strategy to win
the $T$-game in $n+o(n)$ moves, for every $b \le n^{0.005}$. They
also asked for improvements of the parameter $b$ (regardless of the
number of moves required for Maker to win). In this paper, as a
third application of our main result, we show how to obtain such an
improvement for a large family of trees. Those are trees $T$ with
$\Delta(T)=O(1)$ containing a \emph{bare path} of length
$\Theta(n)$, where a bare path is a path for which all the interior
vertices are of degree exactly two in $T$. In fact we prove the
following much stronger result:

\begin{theorem}\label{thm:appTrees}
For every $\alpha>0$ and $\Delta>0$ there exists a
$\delta:=\delta(\alpha,\Delta)>0$ such that for every sufficiently large
integer $n$ the following holds. For $b\le \frac{\delta n}{\log
n}$, in the $(1:b)$ Maker-Breaker game played on $E(K_n)$, Maker has
a strategy to build a graph which contains copies of all the
spanning trees $T$ such that:
\begin{enumerate} [(i)]
\item $\Delta(T)\le \Delta$, and
\item $T$ has a bare path of length at least $\alpha n$.
\end{enumerate}
\end{theorem}

\begin{remark}
Note that the bias $b$ in \thmref{thm:appTrees} is
tight (up to a constant factor), as Chv\'atal and Erd\H{o}s
showed \cite{CE} that for $b=\lfloor\frac{(1+\eps)n}{\ln n}\rfloor$
Breaker can isolate a vertex in Maker's graph.
\end{remark}

The rest of the paper is organized as follows. In \secref{section:aux}
we present some auxiliary results. In \secref{section:main} we prove
\thmref{thm:main}, and in \secref{section:applications} we show how to apply
\thmref{thm:main} combined with local resilience statements
(introduced in Subsection~\ref{sec:resilience}) to various games.

\subsection{Notation}
A graph $G$ is given by a pair of its (finite) vertex set
$V(G)$ and edge set $E(G)$. For a subset $X$ of vertices, we use
$E(X)$ to denote the set of edges spanned by $X$, and for
two disjoint sets $X,Y$, we use $E(X,Y)$ to denote the number
of edges with one endpoint in $X$ and the other in $Y$.
Let $G[X]$ denote the subgraph of $G$ induced by a subset of
vertices $X$. We write $N(X)$ to denote the collection of vertices
that have at least one neighbor in $X$.
When $X$ consists of a single vertex, we abbreviate $N(v)$
for $N(\{v\})$, and let $d(v)$ denote the cardinality of $N(v)$,
i.e., the degree of $v$. Moreover, if $X$ is a set of vertices,
we let $G\setminus X$ to be the induced subgraph $G[V(G)\setminus X]$.
When there are several graphs under consideration, we use subscripts
such as $N_G(X)$ indicating the relevant graph of interest.

To simplify the presentation, we often omit
floor and ceiling signs whenever these are not crucial and make no
attempts to optimize the absolute constants involved. We also assume
that the parameter $n$ (which always denotes the number of vertices
of the host graph) tends to infinity and therefore is sufficiently
large whenever necessary. All our asymptotic notation symbols
($O$, $o$, $\Omega$, $\omega$, $\Theta$) are relative to this
variable $n$.

\section{Auxiliary results}
\label{section:aux}
In this section we present some auxiliary results that will be used
throughout the paper.

\subsection{Binomial distribution bounds}
We use extensively the following well-known bound on the lower and
the upper tails of the Binomial distribution due to Chernoff (see,
e.g., \cite{AloSpe2008}).

\begin{lemma}\label{Che}
If $X \sim \Bin(n,p)$, then
\begin{itemize}
  \item $\Prob\left[X<(1-a)np\right]<\exp\left(-\frac{a^2np}{2}\right)$
        for every $a>0.$
  \item $\Prob\left[X>(1+a)np\right]<\exp\left(-\frac{a^2np}{3}\right)$
        for every $0 < a \leq 1.$
\end{itemize}
\end{lemma}

\noindent The following is a trivial yet useful bound.
\begin{lemma}\label{l:Che}
Let $X \sim \Bin(n,p)$ and $k \in \mathbb{N}$. Then
\[
  \Prob[X\ge k] \le \left(\frac{enp}{k}\right)^k.
\]
\end{lemma}
\begin{proof}
$\Prob[X \ge k] \le \binom{n}{k}p^k \le
\left(\frac{enp}{k}\right)^k$.
\end{proof}

\subsection{The MinBox game}

Consider the following variant of the classical \emph{Box Game}
introduced by Chv\'atal and Erd\H{o}s in \cite{CE}, which we refer
to as the \emph{MinBox game}. The game $MinBox(n,D,\alpha,b)$ is a
$(1:b)$ Maker-Breaker game played on a family of $n$ disjoint sets
(\emph{boxes}), each having size at least $D$. Maker's goal is to
claim at least $\alpha |F|$ elements from each box $F$. In the proof
of our main result, we make use of a specific strategy $S$ for Maker
in the MinBox game. This strategy not only ensures his victory, but
also allows Maker to maintain a reasonable proportion of elements in
all boxes throughout the game.

Before describing the strategy, we need to introduce some notation.
Assume that a MinBox game is in progress, let $w_M(F)$ and $w_B(F)$
denote the number of Maker's and Breaker's current elements in box
$F$, respectively. Furthermore, let $\danger(F):=w_B(F)-b\cdot
w_M(F)$ be the \emph{danger value} of $F$. Finally, we say that a
box $F$ is \emph{free} if it contains an element not yet claimed by
either player, and it is \emph{active} if $w_M(F) < \alpha |F|$.
Maker's strategy is as follows:

\textbf{Strategy $S$:} In any move of the game, Maker identifies one
free active box having maximal danger value (breaking ties
arbitrarily), and claims one arbitrary free element from it.

We are ready to state the following theorem.


\begin{theorem} \label{thm:MinBoxGame}
Let $n$, $b$, and $D$ be positive integers, and $0<\alpha<1$. Assume
that Maker plays the game $MinBox(n,D,\alpha,b)$ according to the
strategy $S$ described above. Then he ensures that, throughout the
game, every active box $F$ satisfies
\[
  \danger(F)\le b(\ln n+1).
\]
In particular, if $\alpha < \frac{1}{1+b}$ and $D\ge\frac{b(\ln n+1)}{%
1-\alpha(b+1)}$, then $S$ is a winning strategy for Maker in this
game.
\end{theorem}
The proof of this result can be found in the Appendix. We remark
that it is very similar to the proof of Theorem 1.2 in \cite{GS}.

\subsection{Local resilience}
\label{sec:resilience}

In this subsection we describe several results related to local
resilience of monotone graph properties. The main result of this
paper (\thmref{thm:main}) shows a connection between local
resilience of graphs and Maker-Breaker games, therefore, in order to
be able to apply it, we first need to present some results related
to local resilience of various properties of random graphs.

The first statement of this section is a theorem from \cite{LS}
providing a good bound on the local resilience of a random graph
with respect to the property ``being Hamiltonian''. This result will
be used in the proof of \thmref{thm:app1} for the Hamiltonicity
game. We remark that for our purposes, prior (and weaker) results on
the local resilience of a random graph with respect to Hamiltonicity
(for example those in \cite{FK}) would suffice.

\begin{theorem}[Theorem 1.1, \cite{LS}] \label{thm:resHam}
For every positive $\eps>0$, there exists a constant
$C=C(\eps)$ such that for $p\ge \frac{C\ln n}{n}$, a graph
$G\sim \gnp(n,p)$ is \aas such that the following holds. Suppose that
$H$ is a subgraph of $G$ for which $G'=G\setminus H$ has minimum
degree at least $(1/2+\eps)np$, then $G'$ is Hamiltonian.
\end{theorem}

The following result from \cite{KLS} is related to the local resilience
of a typical $G\sim \gnp(n,p)$ with respect to pancyclicity.

\begin{theorem}[Thereom 1.1, \cite{KLS}] \label{thm:resPan}
If $p=\omega(n^{-1/2})$, then the local resilience of $G\sim \gnp(n,p)$ with
respect to the property ``being pancyclic" is \aas $1/2+o(1)$.
\end{theorem}

The following theorem shows that a sparse random graph $G\sim
\gnp(n,p)$ is typically such that even if one deletes a small fixed
fraction of edges from each vertex $v\in V(G)$, it still contains a
copy of every tree $T$ having a bare path of linear length and
having bounded maximum degree. This result relates to the local
resilience of the property of being \emph{ universal} for this
particular class of trees, and it is an essential component in the
proof of \thmref{thm:appTrees}.

\begin{theorem} \label{thm:resTrees}
For every $\alpha>0$ and $\Delta>0$, there exist $\eps>0$ and
$C_0$ such that for every $p\ge C_0\ln n/n$, $G\sim\gnp(n,p)$ is \aas
such that the following holds. For every subgraph $H\subseteq G$
with $\Delta(H)\le \eps np$, the graph $G'=G\setminus H$
contains copies of all spanning trees $T$ such that:
\begin{enumerate}[$(i)$]
\item $\Delta(T)\le \Delta$, and
\item $T$ contains a bare path of length at least $\alpha n$.
\end{enumerate}
\end{theorem}

In order to prove \thmref{thm:resTrees} we need the following
theorem due to Balogh, Csaba and Samotij \cite{BCS} about the local
resilience of random graphs with respect to the property
``containing all the almost spanning trees with bounded degree''.

\begin{theorem}[Theorem 2, \cite{BCS}] \label{thm:resAlmostTrees}
Let $\beta$ and $\gamma$ be positive constants, and assume that
$\Delta\ge 2$. There exists a constant $C_0=C_0(\beta,\gamma,\Delta)$ such
that for every $p\ge C_0/n$, a graph $G\sim \gnp(n,p)$ is \aas such
that the following holds. For every subgraph $H$ of $G$ for which
$d_H(v)\le (1/2-\gamma)d_G(v)$ for every $v\in V(G)$, the graph
$G'=G\setminus H$ contains all trees of order at most $(1-\beta)n$
and maximum degree at most $\Delta$.
\end{theorem}

\begin{proof}[Proof of \thmref{thm:resTrees}]
Let $\alpha>0$ and $\Delta>0$ be two positive constants. Let
$\eps:=\eps(\alpha)>0$ be a sufficiently small constant and let
$C_0=C_0(\eps,\Delta)>0$ be a sufficiently large constant. Let
$G\sim \gnp(n,p)$ be a random graph, $H\subseteq G$ be any subgraph
with $\Delta(H)\le \eps np$ and denote $G'=G\setminus H$. We wish to
show that $G'$ contains a copy of every spanning tree $T$ which
satisfies $(i)$ and $(ii)$. This can be done as follows. Assume that
$G$ has been generated by a two-round-exposure and is presented as
$G=G_1\cup G_2$, where $G_1\sim \gnp(n,p/2)$, $G_2 \sim \gnp(n,q)$,
and $q$ is a positive constant for which $1-p=(1-p/2)(1-q)$. Observe
that $q>p/2$. Let $V_0$ be a random subset of $V(G)$ of size
$|V_0|=0.99\alpha n$ and denote $G'_1=G_1 \setminus V_0$. Note that
$G'_1\sim \gnp((1-0.99\alpha)n,p/2)$ and that \aas $d_{G'_1}(v)\ge
(1-0.99\alpha- \eps)np/2$ for every $v\in V(G'_1)$ (this can be
easily shown using Chernoff's inequality, choosing $C_0$
appropriately, and applying the union bound). In addition, note that
for every $v\in V(G)$, the degree of $v$ into $V_0$ (in $G_1$) is at
least $(0.99\alpha-\eps)np/2$. Let $T$ be a tree which satisfies
$(i)$ and $(ii)$, and let $P=v_0v_1 \ldots v_t$ be a bare path of
$T$ with $t=\alpha n$. Let $T'$ be the tree obtained from $T$ by
deleting $v_1,\ldots,v_{t-1}$ and adding the edge $v_0v_t$. Note
that $|V(T')|=(1-\alpha) n+1$.

Let $\beta$ be such that $(1-\beta)|V(G'_1)|=|V(T')|$. Applying
\thmref{thm:resAlmostTrees} to $G'_1$, with (say) $\gamma=1/4$
and $\beta$, using the fact that $\eps$ is sufficiently small
we conclude that there exists a copy $T''$ of $T'$ in $G'_1\setminus
H$. Let $x$ and $y$ denote the images (in $T''$) of $v_0$ and $v_t$
(from $T'$), respectively.

Let $V'=(V(G)\setminus V(T''))\cup\{x,y\}$. In order to complete the
proof, we should be able to show that $(G\setminus H)[V']$ contains
a Hamilton path with $x$ and $y$ as its endpoints. Note that
$V_0\subseteq V'$ and that $V'\setminus V_0$ and the two designated
vertices $x$ and $y$ heavily depend on the tree $T$ which we are
trying to embed. Therefore, we wish to show that $G$ is \aas such
that for every possible option for $V'$ (with two designated
vertices $x$ and $y$), $(G\setminus H)[V']$ contains a Hamilton path
with $x$ and $y$ as its endpoints. For this, note that \aas
$\delta\left(\left(G_1\setminus H\right)[V']\right)\ge 0.491\alpha np$.
Indeed, as we previously remarked, \aas every $v\in V(G)$
has degree (in $G_1$) at least $(0.99\alpha-\eps)np/2$ into
$V_0$. Since $\eps$ is small enough, it follows that
$\delta(G_1[V'])\ge (0.99\alpha-\eps)np/2\ge 0.494\alpha np$.
Combining the last inequality with the fact that $\Delta(H)\le
\eps np$, we obtain $\delta((G_1\setminus H)[V'])\ge
(0.494\alpha-\eps)np\ge 0.491 \alpha np$.
Now, using the following claim (will be proven later) we deduce that
a graph $G_1 \sim \gnp(n,p/2)$ is \aas such that any subgraph
$D\subseteq G_1$ on $\alpha n+1$ vertices with $\delta(D)\ge 0.49\alpha np$
has ``good" expansion properties (our candidate for $D$ will be
$\left(G_1\setminus H\right)[V']$, and we assume that
$\eps<0.001\alpha$).

\begin{claim}\label{claim11}
A graph $G_1\sim \gnp(n,p/2)$ (where $p$ is the same as in
\thmref{thm:resTrees}) is \aas such that for any
subgraph $D\subset G_1$ with $|V(D)|=\alpha n+1$ and with
$\delta(D)\ge 0.49\alpha np$, the following holds:
\[
  |N_{D}(X)\setminus X|\ge 2|X|+2
\]
for every $X\subseteq V(D)$ with $|X|\le |V(D)|/5$.
\end{claim}

Next, we show how to use the edges of $G_2$ in order to turn the graph
$(G_1\setminus H)[V']$ into a graph which contains a Hamilton path connecting
$x$ and $y$. A routine way to turn a non-Hamiltonian graph $D$ that satisfies
some expansion properties (as in \clref{claim11}) into a
Hamiltonian graph is by using \emph{boosters}. A booster is a
non-edge $e$ of $D$ such that the addition of $e$ to $D$ creates a
path which is longer than a longest path of $D$, or
 turns $D$ into a Hamiltonian graph. In order to turn $D$ into a
Hamiltonian graph, we start by adding a booster $e$ of $D$. If the
new graph $D\cup \{e\}$ is not Hamiltonian then one can continue by
adding a booster of the new graph. Note that after at most $|V(D)|$
successive steps the process must terminate and we end up with a
Hamiltonian graph. The main point using this method is that it is
well-known (for example, see \cite{BolRand}) that a non-Hamiltonian
graph $D$ with ``good" expansion properties has many boosters.
However, our goal is a bit different. We wish to turn $D$ into a
graph that contains a Hamilton path with $x$ and $y$ as its
endpoints. In order to do so, we add one (possibly) fake edge $xy$
to $D$ and try to find a Hamilton cycle that contains the edge $xy$.
Then, the path obtained by deleting this edge from the Hamilton
cycle will be the desired path. For that we need to define the
notion of \emph{$e$-boosters}.

Given a graph $D$ and a pair $e\in \binom{V(D)}{2}$, consider a
path $P$ of $D\cup \{e\}$ of maximal length which contains $e$ as an
edge. A non-edge $e'$ of $D$ is called an $e$-booster if
$D\cup\{e,e'\}$ contains a path $P'$ which passes through $e$ and
which is longer than $P$, or that $D\cup\{e,e'\}$ contains a
Hamilton cycle that uses $e$. The following lemma shows that every
connected and non-Hamiltonian graph $D$ with ``good" expansion
properties has many $e$-boosters for every possible $e$.

\begin{lemma} \label{lem:boosters}
Let $D$ be a connected graph for which $|N_{D}(X)\setminus X|\ge
2|X|+2$ holds for every subset $X\subseteq V(D)$ of size $|X|\le
k$. Then, for every pair $e\in \binom{V(D)}{2}$ such that
$D\cup\{e\}$ does not contain a Hamilton cycle which uses the edge
$e$, the number of $e$-boosters for $D$ is at least $(k+1)^2/2$.
\end{lemma}

The proof of the previous lemma is very similar to the proof of the
well-known P\'osa's lemma using the ordinary boosters (\cite{FK},
Lemma 4), and hence we postpone it to the appendix. The only
difference is that in the proof of \lemref{lem:boosters} we
forbid rotations that destroy the edge $e$; and so the number of
possible rotations with a given fixed endpoint drops by at most two.

Note that in order to turn $(G_1\setminus H)[V']$ into a graph that
contains a Hamiltonian cycle passes through $e$, one should
repeatedly add $e$-boosters, one by one, at most $|V'|=\alpha n$
times. Therefore, in order to complete the proof, it is enough to
show that \aas a graph $G_2\sim \gnp(n,q)$ is such that
$G_2\setminus H$ contains ``many" $e$-boosters for any graph
obtained from $(G_1\setminus H)[V']$ by adding a set of edges $E_0$
of size at most $\alpha n$. In the following lemma we formalize and
prove this statement. This is the final ingredient in the proof of
\thmref{thm:resTrees}.

\begin{lemma}\label{GnpContainsManyBoosters}
Assume that $G_1$ satisfies the conclusion of \clref{claim11}. Then
$G_2\sim \gnp(n,q)$ is \aas such that the following holds. Suppose that
\begin{enumerate}[$(i)$]
\item $V'\subseteq V(G_1)$ is a subset of size $|V'|=\alpha n+1$,
\item $H\subset G_1$ is a subgraph such that $\delta((G_1\setminus H)[V'])\ge 0.49\alpha np$,
\item $e=xy\in \binom{V'}{2}$, and
\item $E_0$ is a subset of at most $\alpha
n$ pairs of $V'$.
\end{enumerate}
Then, $G_{E_0,e,V',H}=(G_1\setminus H)[V']\cup E_0\cup \{e\}$ contains
a Hamilton cycle that passes through $e$, or $E(G_2)$ contains at
least $\alpha^2n^2p/200$ $e$-boosters for $G_{E_0,e,V',H}$.
\end{lemma}

\begin{proof}
Since $G_1$ satisfies the conclusion of \clref{claim11}, by
combining it with \lemref{lem:boosters}, it follows that for every
$V'\subseteq V(G_1)$ of size $\alpha n+1$ and $e\in \binom{V'}{2}$,
for every subset $E_0$ of at most $\alpha n$ pairs of $V'$, for
every $H$ with $\delta(G_1\setminus H)[V']\ge 0.49\alpha np$, the graph
$G_{E_0,e,V',H}$ has at least $\alpha^2n^2/50$ $e$-boosters. Fix such
$V'$, $e$, $E_0$ and $H$, and observe that the expected number of
$e$-boosters in $G_2$ is at least $(\alpha^2n^2/50)\cdot q\geq
\alpha^2n^2p/100$ (recall that $q>p/2$). Therefore, by Chernoff's
inequality (\lemref{Che}) it follows that the probability for
$E(G_2)$ to have at most $\alpha^2n^2p/200$ $e$-boosters for
$G_{E_0,e,V',H}$ is at most $\exp(-Cn^2p)$, where $C$ is a constant
which depends only on $\alpha$. Applying the union bound, running
over all the options for choosing $V'$, $H$, $e$ and $E_0$, we
obtain that the probability for having such $V'$, $e$, $H$ and $E_0$ for
which $G_2$ contains at most $\alpha^2n^2p/200$ $e$-boosters for
$G_{E_0,e,V',H}$ is at most
\[
  \sum_{t=1}^{\eps n^2p}2^n\binom{e(G_1[V'])}{t}n^2
  \binom{\alpha^2 n^2}{\alpha n}\exp(-Cn^2p)\le
\]
\[
  2^n\eps n^4p\left(\frac{e\alpha^2n^2p}{\eps n^2p}\right)^{\eps n^2p}
  \left(e\alpha n\right)^{\alpha n}\exp(-Cn^2p)=o(1).
\]
where the last inequality holds for $\eps$ which is much smaller
than $\alpha$ and for $p\geq C_0\ln n/n$ where $C_0$ is sufficiently
large. This completes the proof.
\end{proof}

Before we complete the proof of \thmref{thm:resTrees}, we prove
\clref{claim11}.

\begin{proof}[Proof of \clref{claim11}.]
Let $S\subseteq V(G_1)$ be any subset of vertices of size $|S|\le \frac{2\alpha
n}{\sqrt{\ln n}}$, and note that $|E(G_1[S])|\sim
\Bin(\binom{|S|}{2},p/2)$. Therefore, using \lemref{l:Che}
we obtain that
\begin{align}
  \Prob[|E(G_1[S])|\ge |S|np/\ln\ln n] & \le
  \left(\frac{e|S|^2p\ln\ln n}{4|S|np}\right)^{|S|np/\ln\ln n}\nonumber \\
  &=\left(\frac{e|S|\ln\ln n}{4n}\right)^{|S|np/\ln\ln n}.
\end{align}

Let $\mathcal E$ denote the event ``there exists a subset $S\subseteq V(G_1)$ of size $|S|\le \frac{2\alpha
n}{\sqrt{\ln n}}$ for which $|E(G_1[S])|\ge |S|np/\ln\ln n$". By applying the union bound and the estimate $(1)$ we obtain that

\begin{align}
\Pr\left[\mathcal E\right]=&\sum_{s=1}^{\frac{2\alpha n}{\sqrt{\ln
n}}}\binom{n}{s}\left(\frac{es\ln\ln n}{4n}\right)^{snp/\ln\ln
n} \nonumber \\
&\leq \sum_{s=1}^{\frac{2\alpha n}{\sqrt{\ln
n}}}\left(\frac{en}{s}\right)^s \left(\frac{es\ln\ln
n}{4n}\right)^{snp/\ln\ln n}=o(1).
\end{align}

Now, let $D\subseteq G_1$ be a subgraph on $\alpha n+1$ vertices
with $\delta(D)\ge 0.49\alpha np$, and we wish to show that for
every $X\subseteq V(D)$, if $|X|\le |V(D)|/5$, then
$|N_D(X)\setminus X|\ge 2|X|+2$. First, we consider the case $|X|\le
\frac{\alpha n}{3\sqrt{\ln n}}$. Assume that there exists a subset
$X\subseteq V(G_1)$ in this range for which $|N_{D}(X)\setminus
X|\le 2|X|+1$. Using the fact that $\delta(D)=\Theta(np)$ we obtain
$|E(D[X\cup N_D(X)])|\ge |X|\cdot\Theta(np)$. Now, since $|E(D[X\cup
N_D(X)])|\le |E(G_1[X\cup N_D(X)])|$ and since $|X\cup N_D(X)|\le
\frac{\alpha n}{\sqrt{\ln n}}+2<\frac{2\alpha n}{\sqrt{\ln n}}$, by $(2)$ it happens with probability $o(1)$. Therefore, we conclude that \aas
$|N_{D}(X)\setminus X|\ge 2|X|+2$ holds for every subset $X\subseteq
V(D)$ of size at most $\frac{\alpha n}{3\sqrt{\ln n}}$.

Second, assume $\frac{\alpha n}{3\sqrt{\ln n}}<|X|\le |V(D)|/5$. In
this range it is enough to show that \aas for every two disjoint
subsets of vertices $X,Y\subseteq V(D)$ of sizes $|X|=\frac{\alpha
n}{3\sqrt{\ln n}}$ and $|Y|=\alpha n/10$ we have $|E_D(X,Y)|\neq 0$.
Indeed, let $X\subseteq V(D)$ be a subset in this range and assume
that $|N_D(X)\setminus X|\le 2|X|+2$. In particular, since $|X|\le
|V(D)|/5$ we conclude that $|X\cup N_D(X)|\le |X|+2|X|+2\le
4|V(D)|/5$. Therefore, one can find $X'\subseteq X$ of size
$|X'|=\frac{\alpha n}{3\sqrt{\ln n}}$ and $Y\subseteq V(D)\setminus
\left(X\cup N_D(X)\right)$ of size $|Y|=\alpha n/10$ for which
$|E_D(X,Y)|=0$, a contradiction.

In order to show that the above mentioned property \aas holds, let
$X,Y\subseteq V(G_1)$ be two disjoint subsets of sizes
$|X|=\frac{\alpha n}{3\sqrt{\ln n}}$ and $|Y|=\alpha n/10$. For a
vertex $x\in X$, let $d_{G_1}(x,Y)$ denote the number of neighbors
of $x$ in $Y$, and observe that $d_{G_1}(x,Y)\sim \Bin(|Y|,p/2)$.
Therefore, the probability that $d_{G_1}(x,Y)$ is outside the
interval $(|Y|p/3,2|Y|p)$ is at most
$e^{-\Theta(|Y|p|)}=e^{-\Theta(np)}$. Since all the events
$d_{G_1}(x',Y)\notin (|Y|p/3,2|Y|p)$ ($x'\in X$) are mutually
independent, we conclude that
$$\Prob[\textrm{ for at least } \frac{\ln\ln n}p \textrm{ vertices } x\in X \textrm{ we have } d_{G_1}(x,Y)\notin
(|Y|p/3,2|Y|p)]$$
$$\leq \binom{n}{\ln\ln n/p}e^{-\Theta(np\cdot
\ln\ln n/p)}=e^{-\Theta(n\ln\ln n)}.$$

Now, by applying Chernoff's inequality and the union bound, we obtain that the
probability for having two such sets $X$ and $Y$ such that for at
least $\frac{\ln\ln n}p$ vertices $x\in X$ we have
$d_{G_1}(x,Y)\notin (|Y|p/3,2|Y|p)$, is at most

\[
  \binom{n}{\alpha n/(3\sqrt{\ln n})}\binom{n}{\alpha n/10}
  e^{-\Theta(n\ln\ln n)}\le 4^ne^{-\Theta(n\ln\ln n)}=o(1).
\]

Next, we wish to show that \aas in $G_1$, there are at most
$\frac{\ln\ln n}p$ vertices $v\in V(D)$ with $d_{G_1}(v,V(D))\notin
(0.99|V(D)|p/2,1.01|V(D)|p/2)$. This can be done in the following way: for each subset of vertices $D\subset V(G_1)$ of size $\alpha n+1$, we fix an arbitrary orientation of the complete graph induced by $V(D)$ which is as regular as possible and consider $G_1[V(D)]$ as an \emph{oriented} graph. Now, note that clearly the event `` there are at most
$\frac{\ln\ln n}p$ vertices $v\in V(D)$ with $d_{G_1}(v,V(D))\notin
(0.99|V(D)|p/2,1.01|V(D)|p/2)$" is contained in the event $\mathcal E'=$ `` there are at most
$\frac{\ln\ln n}p$ vertices $v\in V(D)$ with at least one of $d^+_{G_1}(v,V(D))$ or $d^-_{G_1}(v,V(D))$ not in
$(0.99|V(D)|p/4,1.01|V(D)|p/4)$". Let $\sigma\in \{+,-\}$, $D$ and $v\in V(D)$, and note that $d^{\sigma}_{G_1}(v,V(D))\sim \Bin(|V(D)|/2,p/2)$ and that the random variables $\{d_{G_1}^{\sigma}(v,V(D)):v\in V(D)\}$ are mutually independent. Therefore, using a similar calculation as before and taking the union bound over $\sigma\in \{+,-\}$ we obtain the claim. Assuming this, let $X$ and $Y$ be two
subsets of sizes $|X|=\frac{\alpha n}{3\sqrt{\ln n}}$ and
$|Y|=\alpha n/10$. By the above mentioned arguments, there exists a
vertex $x\in X$ with $d_{G_1}(x,Y)\in (|Y|p/3,2|Y|p)$ and
$d_{G_1}(x,V(D))\in (0.99|V(D)|p/2,1.01|V(D)|p/2)$. Now, since $\delta(D)\ge
0.49\alpha np=0.98 \alpha np/2$, it follows that there are at most
$0.03\alpha np/2$ edges touching $x$ in $G_1$ which do not appear in
$V(D)$. Since $d_{G_1}(x,Y)\geq |Y|p/3\geq \alpha np/30$, it follows
$D$ still contains edges between $x$ and $Y$, and therefore
$|E_D(X,Y)|\neq 0$. All in all, we conclude that $|N_{D}(X)\setminus
X|\ge 2|X|+2$ holds for every $|X|\le |V(D)|/5$. This completes the
proof.
\end{proof}

This also completes the proof of Theorem \ref{thm:resTrees}.
\end{proof}

\section{Proof of the main result}
\label{section:main}

\begin{proof}[Proof of \thmref{thm:main}]
In order to prove the theorem, we provide Maker with a random
strategy that enables him to generate a random graph $G'\sim
\gnp(G,p)$, and \aas claim at least $1-\eps$ fraction of the
edges of $G'$ touching each vertex. We then use the fact that
$\mc P$ is $(G,p,\eps)$-resilient to conclude that $G'$
\aas satisfies $\mc P$. Note that since a Maker-Breaker game
is deterministic, and since the strategy we describe \aas ensures
Maker's win against any strategy of Breaker, it follows that Maker
also has a deterministic winning strategy.

We now present the random strategy for Maker. In this strategy,
Maker will gradually generate a random graph $G'\sim \gnp(G,p)$,
by tossing a biased coin on each edge of $G$, and declaring
that it belongs to $G'$ independently with probability $p$. Each
edge which Maker has tossed a coin for is called \emph{exposed}, and
we say that Maker is \emph{exposing} an edge $e \in E(G)$ whenever
he tosses a coin to decide about the appearance of $e$ in $G'$. To
keep track of the unexposed edges, Maker maintains a set
$U_v\subseteq N_G(v)$ of the \emph{unexposed neighbors} of $v$, for
each vertex $v$ in $G$; i.e. $u\in U_v$ if and only if the edge $vu$
remains to be exposed. Initially, $U_v=N_G(v)$ for all $v\in V(G)$.
We remark that Maker will expose all edges of $G$, even those that
belong to Breaker.

In every turn, Maker chooses an \emph{exposure vertex} $v$ (we will
later discuss the choice of the exposure vertex) and starts to
expose edges connecting $v$ to vertices in $U_v$, one by one in an
arbitrary order, until one edge in $G'$ is found (that is, until he
has a first success). If this exposure happens to reveal an edge
$vu\in E(G')$ not yet claimed by Breaker, Maker claims it and
completes his move. Otherwise, either the exposure failed to reveal
a new edge in $G'$ (failure of \emph{type I}), or the newly found
edge already belongs to Breaker (failure of \emph{type II}). In
either case, Maker skips his move. Let $f_I(v)$ and $f_{II}(v)$
denote the number of failures of type I and II, respectively, for
the exposure vertex $v$. We remind the reader that Maker's goal is
to make sure that at the end of the game $f_{II}(v)$ is relatively
small, namely, $f_{II}(v) \le \eps d_{G'}(v)$ for all $v\in V(G')$.
We do not know a priori what is the degree of $v$ in $G'$, since
$G'$ is random. However, it is true that \aas
\begin{equation}
  \label{eqn:deg_lower_bound}
  d_{G'}(v) \ge \frac{9}{10} d_G(v)p
\end{equation}
holds for all $v\in V(G)$. To see this, recall that $d_G(v) \ge
\delta(G) \ge \frac{10 \ln n}{\eps p}$, so for any fixed $v\in V(G)$,
\lemref{Che} implies that $\Prob[\Bin(d_G(v),p)<\frac{9}{10} d_G(v)p]=
o(\frac{1}{n})$. Hence, by the union bound, \aas
\eqref{eqn:deg_lower_bound} holds for all vertices in $G$.

In view of \eqref{eqn:deg_lower_bound}, to complete the proof of
\thmref{thm:main} it suffices to show that \aas Maker can ensure
that $f_{II}(v)\le \frac{9}{10}\eps d_G(v)p$ for all vertices $v\in
V(G)$ at the end of the game. Since Maker's goal here is to build a
random graph, if a failure of type I occurs it does not harm Maker.

To keep the failures of type II under control, concurrently to the
game played on $G$, we simulate a game $MinBox(n, 4\delta(G), p/2,
2b)$. In this simulated game, there is one box $F_v$ for each $v\in
V(G)$ which helps us to keep track of the exposure of edges touching
$v$. Initially, we set the sizes of the boxes as $|F_v|=4d_G(v)$.
Now, we describe Maker's strategy.

\noindent \textbf{Maker's strategy $S_M$:} Maker's strategy is
divided into the following two stages.

\textbf{Stage 1:} Before his move, Maker updates the status of the
simulated game by pretending that Breaker claimed one free element
from both $F_v$ and $F_u$, for each edge $vu$ occupied in Breaker's
last move. Maker then identifies a free active box $F_v$ having
highest danger value in the simulated game (breaking ties
arbitrarily). If there is no such box, Maker proceeds to the second
stage of the strategy. Otherwise, let $F_v$ be such a box. Maker
claims one free element from $F_v$, and selects $v$ as the exposure
vertex. Let $\sigma: [m]\rightarrow U_v$ be an arbitrary permutation
on $U_v$, where $m:=|U_v|$. Maker starts tossing a biased coin for
vertices in $U_v$, independently at random, according to the
ordering of $\sigma$.
\begin{enumerate} [$(a)$]
\item If there were no successes, then Maker declares this turn
as a \emph{failure of type I}, thereby incrementing $f_{I}(v)$, and
skips his move in the original game. Maker then claims $\lceil\frac
p2 \cdot |F_v|\rceil - 1$ additional free elements from $F_v$ (or
all the remaining free elements of $F_v$ if there are not enough
such elements) in the simulated game, and updates $U_v:=\emptyset$,
and $U_{\sigma(i)}:=U_{\sigma(i)} \setminus \{v\}$ for each $i\le
m$.
\end{enumerate}
Assume that Maker's first success has happened at the $k$th coin tossing.
\begin{enumerate} [$(a)$]
\setcounter{enumi}{1}
\item If the edge $v\sigma(k)$ is not free, then Maker declares $v\sigma(k)$
as a \emph{failure of type II}, increments $f_{II}(v)$ by one, and
skips his move in the original game. Maker then updates
$U_v:=U_v\setminus \{\sigma(i): i\le k\}$, and
$U_{\sigma(i)}:=U_{\sigma(i)} \setminus \{v\}$ for each $i\le k$.
\item Otherwise, Maker claims the edge $v\sigma(k)$. In this case Maker also
claims a free element from box $F_{\sigma(k)}$ and then updates
$U_v:=U_v\setminus \{\sigma(i): i\le k\}$, and
$U_{\sigma(i)}:=U_{\sigma(i)} \setminus \{v\}$ for each $i\le k$.
\end{enumerate}

\textbf{Stage 2:} In this stage, there are no free active boxes. Let
$U := \{vu: v\in V(G), u\in U_v\}$. For
each $e=vu\in U$, Maker declares a failure of type II on both $u$ and $v$
(i.e., increments both $f_{II}(u)$ and $f_{II}(v)$ by one) with probability
$p$, independently at random. After the end of this stage,
Maker stops playing the game altogether, and skips all his subsequent moves.

We now prove that by following $S_M$, Maker typically achieves his
goal. For the sake of notation, at any point during the game, we
denote by $d_M(v)$ and $d_B(v)$ the degrees of $v$ in the subgraphs
currently occupied by Maker and Breaker, respectively. The proof
will follow from the next four claims.

In the first claim, we prove that no box in the simulated game is ever
exhausted of free elements.  This implies that Maker is always able to
effectively simulate all the moves in the original game, that is,
the moves from Breaker (which in the simulated game causes two
elements to be claimed), and also the moves from Maker. In particular,
whenever a failure of type I occurs, Maker will claim exactly
$\lceil \frac{p}{2} \cdot |F_v|\rceil-1$ additional free elements
from the relevant box, as Maker's strategy for Stage 1 (case (a))
dictates.
\begin{claim}
\label{claim1}
At any point during the first stage, we have $w_M(F_v)< 1+(1+2p)d_G(v)$
and $w_B(F_v) \le d_G(v)$ for every box $F_v$ in the simulated game.
In particular, $w_M(F_v) + w_B(F_v) \le 4d_G(v)$ thus no box is ever
exhausted of free elements.
\end{claim}
\begin{proof}
Clearly $w_B(F_v) = d_B(v) \le d_G(v)$, and $d_M(v) + f_{II}(v) \le
d_G(v)$. Moreover, $w_M(F_v)=d_M(v) + \lceil \frac{p}{2}|F_v|\rceil
f_I(v) + f_{II}(v)$. We claim that $f_I(v) \le 1$. This is true
because otherwise $F_v$ would still have free elements after the
first failure of type $I$ on $v$, and hence Maker would have claimed
at least $\lceil \frac{p}{2} \cdot |F_v|\rceil$ elements from $F_v$.
This is a contradiction, because $F_v$ would then be inactive,
and thus Maker will never play on $v$ again, which implies $f_{I}(v)
\le 1$. Therefore $w_M(F_v)<1+d_G(v) + \frac{p}{2}\cdot |F_v| \le
1+(1+2p)d_G(v)$, as required.
\end{proof}

\begin{claim} \label{claim2}
For every $v\in V(G)$, $F_v$ becomes inactive before $d_B(v)\ge
\eps d_G(v)/4$.
\end{claim}

\begin{proof}
Let $v\in V(G)$ be any vertex of $V(G)$. Note that in the simulated
game, Maker always claims a free element from one free active box having
highest danger value. This, however, does not imply that Maker exactly follows
the strategy described in \thmref{thm:MinBoxGame}, as he might occasionally
claim more than one free element when a failure of type I occurs.
Nonetheless, we claim that the assertion of \thmref{thm:MinBoxGame} still
holds in this case because of the following reason. If Maker has a winning
strategy in a $(1:b)$ Maker-Breaker game, then he also has a winning strategy
in a game in which he is occasionally allowed to claim more than one position
per move. This is due to the monotonic nature of these types of games (recall
that $MinBox$ is a Maker-Breaker game). Hence, by
\thmref{thm:MinBoxGame}, we have
\begin{equation}
  \label{eq:danger}
  \danger(F_v) = w_B(F_v)-2b\cdot w_M(F_v)\le 2b(\ln n+1)
\end{equation}
for every active box $F_v$. Assume that there exists a vertex $v\in
V(G)$ for which $F_v$ is still active and $w_B(F_v)=d_B(v)\ge
\eps d_G(v)/4$. Recall that $b=\lfloor\frac{\eps}{20p}\rfloor$, and
by \eqref{eq:danger} it follows that
\begin{align*}
  w_M(F_v) \ge \frac{w_B(F_v)}{2b} - (\ln n + 1) \ge
  \frac 52 d_G(v)p-(\ln n+1).
\end{align*}
By the assumption that $\delta(G)\ge \frac{10\ln n}{\eps p}$,
we conclude that $w_M(F_v) > 2 d_G(v)p= \frac p2 |F_v|$.
However, since we assumed that $F_v$ is active in $MinBox(n,
4\delta(G), p/2, 2b)$, we must have $w_M(F_v)\le \frac p2 |F_v|$,
which is a contradiction.
\end{proof}

\begin{claim} \label{claim3}
Asymptotic almost surely all edges of $G'$ are exposed before
Stage 2.
\end{claim}
\begin{proof}
Suppose there exists a vertex $v$ at the beginning of the second
stage, such that $U_v \ne \emptyset$. Since $U_v\ne \emptyset$, we
must have $f_{I}(v) = 0$. Moreover, because $F_v$ is not active, we
must also have $w_M(F_v) = d_M(v)+f_{II}(v) \ge \frac p2
|F_v|=2d_G(v)p$. This implies that $d_{G'}(v)\ge d_M(v)+f_{II}(v)\ge
2 d_G(v)p$. Now, since $d_{G'}(v)\sim \Bin(d_G(v),p)$, using
Chernoff's inequality, it follows that
\[
  \Prob[\Bin(d_G(v), p) \ge 2d_G(v)p] \le e^{-d_G(v)p/3} =
  o\left(\frac{1}{n}\right).
\]
Applying the union bound, it thus follows that with probability
$1-o(1)$, there exists no such vertex, proving the claim.
\end{proof}

\begin{claim} \label{claim4}
Asymptotically almost surely, for every $v\in V(G)$ we have
$f_{II}(v)\le \frac{9}{10}\eps d_G(v)p$.
\end{claim}

\begin{proof}
Let $v\in V(G)$ be any vertex. By \clref{claim2}, during Stage 1
Breaker can touch $v$ at most $\eps d_G(v)/4$ times before
$F_v$ becomes inactive. Moreover, by \clref{claim3}, with
probability $1-o(1)$ all the edges of $G'$ were exposed before
the beginning of Stage 2. Since a failure of type II in Stage 1
is equivalent to Maker having a success on one of Breaker's
edges, it follows that $f_{II}(v)$ is stochastically dominated
by $\Bin(m,p)$, where $m = \eps d_G(v)/4$.
Applying \lemref{l:Che} to $f_{II}(v)$ we conclude that the
probability for having more than $\eps d_G(v)p$ edges $vu$
which are failures of type II is at most
\[
  \Prob\left[\Bin(\eps d_G(v)/4,p)\ge \frac{9}{10}\eps d_G(v)p\right]\le
  \left(\frac{e\eps d_G(v)p/4}{\frac{9}{10}\eps d_G(v)p}\right)^{
  \frac{9}{10}\eps d_G(v)p}=
  o\left(\frac 1n\right).
\]
Applying the union bound we obtain that the probability that there
is such a vertex is $o(1)$. Therefore, \aas $f_{II}(v) \le
\frac{9}{10}\eps d_G(v)p$ for all $v\in V(G)$.
\end{proof}

This completes the proof of \thmref{thm:main}.
\end{proof}

\section{Applications}
\label{section:applications}

In this section we show how to apply \thmref{thm:main} in order
to prove Theorems~\ref{thm:app1}, \ref{thm:app2} and \ref{thm:appTrees}.
We also derive a directed graph analog of \thmref{thm:main}.
We start with proving \thmref{thm:app1}, which states that Maker
can win the Hamiltonicity game played on $E(K_n)$ against an asymptotically
optimal (up to a constant factor) bias of Breaker.

\begin{proof}[Proof of \thmref{thm:app1}]
Let $C_1=C(\frac 16)$ be as in \thmref{thm:resHam}, and let
$C_2:=\max\{C_1,1000\}$. First, observe that for $p\ge \frac{C_2\ln
n}{n}$ \aas we have that $G\sim \gnp(n,p)$ satisfies $\delta(G)\ge
\frac 56 np$ (this follows immediately from Chernoff's inequality and the union
bound). Next, note that the property $\mc P:=$``being
Hamiltonian" is $(K_n,p,1/6)$ resilient for $p\ge \frac{C_2\ln
n}{n}$. Indeed, let $H\subseteq G$ be a subgraph for which
$d_H(v)\le \frac 16 d_G(v)$. Observe that in $G':=G-H$ we have
$d_{G'}(v)\ge \frac 56 d_G(v)$. Now, since \aas $\delta(G)\ge
\frac 56 np$, it follows that $\delta(G')\ge \frac{25}{36}np>\frac
23 np$. Therefore, by our choice of $C_2$ and
\thmref{thm:resHam}, it follows that $G'$ is Hamiltonian.

Lastly, applying \thmref{thm:main} with
$\eps=\frac{1}{100}$ (recall that we have an upper bound for $\varepsilon$), $K_n$ (as the host graph $G$),
$p=\frac{C_2\ln n}{n}$ and $\mc P$, we obtain that Maker has a
winning strategy in the $(1:\lfloor\frac{1}{120p}\rfloor)$ game
$\mc P(K_n)$. Note that $\frac{1}{120p}=\frac{n}{120 C_2\ln n}$,
and therefore, by setting $\alpha:=\frac{1}{120C_2}$ we complete
the proof.
\end{proof}

Next, we prove \thmref{thm:app2}

\begin{proof}[Proof of \thmref{thm:app2}]
Let $p=\omega(n^{-1/2})$ and note that by \thmref{thm:resPan}, it
follows that the property $\mc P:=$``being pancyclic" is
$(K_n,p,1/2+o(1))$-resilient. Therefore, by applying \thmref{thm:main}
with $\eps=1/100$, $K_n$ (as the host graph), $p$ and $\mc P$,
we obtain that Maker has a winning strategy in the
$(1:\lfloor\frac{1}{60p}\rfloor)$ game
$\mc P(K_n)$. This completes the proof.
\end{proof}

We turn to prove \thmref{thm:appTrees}.

\begin{proof}[Proof of \thmref{thm:appTrees}]
Let $\alpha>0$ and $\Delta>0$ be two positive constants. Let
$\eps>0$ and $C_0$ be as in \thmref{thm:resTrees}
(applied to $\alpha$ and $\Delta$). Let $C_1\ge
\max\left\{C_0,\frac{20}{\eps}\right\}$ be a large enough
constant for which $G\sim \gnp(n,p)$ \aas satisfies $\Delta(G)\le
(1+\eps)np$, provided that $p = \frac{C_1\ln n}{n}$. Let
$\mc T$ be the set of all trees $T$ on $n$ vertices satisfying:
\begin{enumerate}[(i)]
\item $\Delta(T)\le \Delta$, and
\item $T$ contains a bare path of length at least $\alpha n$,
\end{enumerate}
and let $\mc P$ be the property ``being
$\mc T$-universal" (that is, contains copies of all trees in
$\mc T$). Observe that $\mc P$ is $\left(K_n, p,
\frac{\eps}{1+\eps}\right)$ resilient, and hence
$\left(K_n, p, \frac{\eps}{2}\right)$ resilient for
$p= \frac{C_1 \ln n}{n}$. Indeed, let $H$ be a subgraph of $G$ for
which $d_H(v) \le\frac{\eps}{1+\eps}\cdot d_G(v)$, for
all vertices $v\in V(G)$. Thus $d_H(v)
\le\frac{\eps}{1+\eps}\cdot \Delta(G) \le \eps
n p$. Therefore, by \thmref{thm:resTrees}, $G':=G\setminus H$
satisfies $\mc P$.

Lastly, by applying \thmref{thm:main} with $\min\{\eps/2,1/100\}$ (as
$\eps$), $K_n$ (as the host graph $G$), $p$, and $\mc P$,
we obtain that Maker has a winning strategy in the
$(1:\lfloor\frac{\eps}{40p}\rfloor)$ game $\mc P(K_n)$. By setting
$\delta=\frac{\eps}{40C_1}$, we complete the proof. Note
that we used the fact that $C_1\ge \frac{20}{\eps}$ in order
to verify assumption (iii) in \thmref{thm:main}.
\end{proof}

As a last application, we establish an analog of \thmref{thm:main}
to directed graphs. A \emph{directed} graph $D$ consists of a
set of vertices $V(D)$, and a set of \emph{arcs} (or \emph{directed}
edges) $E(D)$ composed of elements of the form $(u,v)\in
V(D)\times V(D)$, where $u\neq v$. For a directed graph $D$ and a
vertex $v\in V(D)$ we let $d^+(v)$ and $d^-(v)$ denote the out- and
in- degrees of $v$, respectively. Furthermore, we define
$\delta^+(D)$ and $\delta^-(D)$ to be the minimum out- and in-
degrees of $D$, respectively, and set
$\delta^0(D)=\min\{\delta^+(D),\delta^-(D)\}$. Analogously to
graphs, we define $\dnp(D,p)$ to be the model of random
sub-directed graphs of $D$ obtained by retaining each arc of $D$
with probability $p$, independently at random. We write $\dnp(n,p)$
for $\dnp(D,p)$ in the special case where $D$ is the
complete directed graph on $n$ vertices. That is, $V(D)=[n]$ and
$E(D)$ consists of all the possible arcs. Similarly as in
\defref{PisRes}, for a monotone increasing directed graph
property $\mc P$, we say that $\mc P$ is $(D,p,r)$-\emph{resilient}
if the local resilience of $D'\sim \dnp(D,p)$ with respect to
$\mc P$ is at least $r$, where here we mean that by deleting at
each vertex $v$ at most $r\cdot d^+_D(v)$ out- and $r\cdot d^-_D(v)$
in-edges one can obtain a directed graph not having $\mc P$.

\begin{theorem}\label{thm:mainDirected}
For every constant $0<\eps\leq 1/100$ and a sufficiently large
integer $n$ the following holds. Suppose that
\begin{enumerate} [$(i)$]
\item $0< p=p(n)< 1$,
\item $D$ is a directed graph with $|V(D)|=n$,
\item $\delta^{0}(D)\ge \frac{10\ln n}{\eps p}$, and
\item $\mc P$ is a monotone increasing directed graph
property which is $(D,p,\eps)$-resilient.
\end{enumerate}
Then Maker has a winning strategy in the $(1:
\lfloor\frac{\eps}{20p}\rfloor)$ game $\mc P(D)$.
\end{theorem}

\begin{proof} For a directed graph $D$ one can define the following
bipartite graph $G_D$: the parts of $G_D$ are two disjoint copies of
$V(D)$, denoted by $A$ and $B$. For any $a\in A$ and $b\in B$, the
(undirected) edge $ab$ belongs to $E(G_D)$ if and only if the
directed edge $ab$ belongs to $E(D)$. Note that the mapping
$D\rightarrow G_D$ is an injection from the set of all directed
graphs on $n$ vertices to the set of bipartite graphs with two parts
of size $n$ each, and apply \thmref{thm:main} to $G_D$ in the
obvious way. Note that the property $\mc P$ of digraphs
naturally translates to a property $\mc P'$ of bipartite graphs
which is $(G_D,p,\eps)$-resilient.
\end{proof}


\textbf{Acknowledgment.} The authors wish to thank the anonymous
referees for many valuable comments.

\appendix

\section{Proofs of \thmref{thm:MinBoxGame} and \lemref{lem:boosters}}

We begin with the proof of the $MinBox$ game.
\begin{proof}[Proof of \thmref{thm:MinBoxGame}]
The proof of this theorem is very similar to the proof of Theorem
1.2 in \cite{GS}. Since claiming an extra element is never a
disadvantage for any of the players, we can assume that Breaker is
the first player to move. For a subset $X$ of boxes, let
$\avdanan(X)=\frac{\sum_{F\in X}\danger(F)}{|X|}$ denote the average
danger of the boxes in $X$. The game ends when there are no more
free elements left.


First we prove the upper bound for the danger values of active
boxes. Suppose, towards a contradiction, that there exists a
strategy for Breaker that ensures the existence of an active box $F$
satisfying $\danger(F)>b(\ln n+1)$ at some point during the game.
Denote the first time when this happens by $g$. Let
$I=\{F_1,\ldots,F_g\}$ be the set which defines Maker's game, i.e,
in his $i^{th}$ move, Maker plays at $F_i$ for $1\le i\le g-1$ and
$F_g$ is the first active box satisfying $\danger(F_g) > b(\ln n +
1)$. For every $0\le i\le g-1$, let $I_i=\{F_{g-i},\ldots,F_g\}$.
Following the notation of \cite{GS}, let $\danger_{B_i}(F)$ and
$\danger_{M_i}(F)$ denote the danger value of a box $F$,
\emph{directly before} Breaker's and Maker's $i^{th}$ move,
respectively. Notice that in his $g^{th}$ move, Breaker increases
the danger value of $F_g$ to more than $b(\ln n + 1)$. This is only
possible if $\danger_{B_g}(F_g)> b(\ln n+1)-b=b\ln n$.

Analogously to the proof of Theorem 1.2 in \cite{GS}, we state the
following lemmas which estimate the change of the average danger
after a particular move (by either player). In the first lemma we
estimate the changes after Maker's moves.

\begin{lemma} \label{lem:makermindeg}
Let $i$, $1\le i \le g-1$,

$(i)$ if $I_i\neq I_{i-1}$, then $\avdanan_{M_{g-i}}(I_{i}) -
\avdanan_{B_{g-i+1}}(I_{i-1}) \ge 0.$

$(ii)$ if $I_i=I_{i-1}$, then $\avdanan_{M_{g-i}}(I_{i}) -
\avdanan_{B_{g-i+1}}(I_{i-1}) \ge \frac{b}{|I_i|}.$
\end{lemma}
\begin{proof}
For part $(i)$ we have that $F_{g-i}\not\in I_{i-1}$. Since danger
values do not increase during Maker's move, we have
$\avdanan_{M_{g-i}}(I_{i-1}) \ge \avdanan_{B_{g-i+1}}(I_{i-1})$.
Before $M_{g-i}$, Maker selected the box $F_{g-i}$ because its
danger was highest among the active boxes. Thus $\danger(F_{g-i})\ge
\max ( \danger(F_{g-i+1}),\ldots, \danger(F_g))$, which implies
$\avdanan_{M_{g-i}}(I_i)\ge \avdanan_{M_{g-i}}(I_{i-1})$. Combining
the two inequalities establishes part $(i)$.

For part $(ii)$ we have that $F_{g-i}\in I_{i-1}$. In $M_{g-i}$,
$w_M(F_{g-i})$ increases by $1$ and $w_M(F)$ does not change for any
other box $F\in I_{i}$. Besides, the values of $w_B(\cdot)$ do not
change during Maker's move. So $\danger(F_{g-i})$ decreases by $b$,
whereas $\danger(F)$ do not increase for any other box $F\in I_{i}$.
Hence $\avdanan(I_i)$ decreases by at least $\frac{b}{|I_i|}$, which
implies $(ii)$.
\end{proof}

In the second lemma we estimate the changes after Breaker's moves.

\begin{lemma}\label{lem:breakermindeg}
Let $i$ be an integer, $1\le i\le g-1$. Then,

\[
  \avdanan_{M_{g-i}}(I_i) - \avdanan_{B_{g-i}}(I_i) \le
  \frac{b}{|I_i|}.
\]

\end{lemma}
\begin{proof}
The increase of $\sum_{F\in I_{i}} w_B(F)$ during $B_{g-i}$ is at
most $b$. Moreover, since the values of $w_M(F)$ for $F\in I_{i}$ do
not change during Breaker's move, the increase of $\avdanan(I_i)$
(during $B_{g-i}$) is at most $\frac{b}{|I_i|}$, which establishes
the lemma.
\end{proof}

Combining Lemmas~\ref{lem:makermindeg} and~\ref{lem:breakermindeg},
we obtain the following corollary which estimates the change of the
average danger after a full round.

\begin{corollary}\label{coro:danger-change}
Let $i$ be an integer, $1\le i\le g-1$.

$(i)$ if $I_i=I_{i-1}$, then $\avdanan_{B_{g-i}}(I_{i}) -
\avdanan_{B_{g-i+1}}(I_{i-1}) \ge 0.$

$(ii)$ if $I_i\neq I_{i-1}$, then $\avdanan_{B_{g-i}}(I_{i}) -
\avdanan_{B_{g-i+1}}(I_{i-1}) \ge -\frac{b}{|I_i|}$

\end{corollary}

Next, we prove that before Breaker's first move,
$\avdanan_{B_1}(I_{g-1})>0$, thus obtaining a contradiction. To that
end, let $|I_{g-1}|=r$ and let $i_1 < \ldots < i_{r-1}$ be those indices
for which $I_{i_j}\neq I_{i_j-1}$. Note that $|I_{i_j}|=j+1$. Recall
that $\danger_{B_g}(F_g)>b\ln n$, therefore
\begin{eqnarray}\label{eq:lastlinehugecalc}
\avdanan_{B_1}(I_{g-1}) & = & \avdanan_{B_g}(I_{0}) +
\sum_{i=1}^{g-1} \left( \avdanan_{B_{g-i}}(I_{i}) -
\avdanan_{B_{g-i+1}}(I_{i-1}) \right)
\nonumber \\
& \ge & \avdanan_{B_g}(I_{0}) + \sum_{j=1}^{r-1} \left(
\avdanan_{B_{g-i_j}}(I_{i_j}) - \avdanan_{B_{g-i_j+1}}(I_{i_j-1})
\right) \enspace\enspace \mbox{[by
\corref{coro:danger-change}$(i)$]}
\nonumber \\
& \ge & \avdanan_{B_g}(I_{0}) - \sum_{j=1}^{r-1} \frac{b}{j+1}
\enspace\enspace \mbox{[by
\corref{coro:danger-change}$(ii)$]}
\nonumber \\
& \ge & \avdanan_{B_g}(I_{0}) - b\ln n \nonumber > 0,
\end{eqnarray}
and this contradiction establishes the upper bound for the danger
values of active boxes.

Lastly, consider a $MinBox(n,D,\alpha, b)$ game where $\alpha <
\frac{1}{b+1}$ and $D\ge \frac{b(\ln n + 1)}{1 -\alpha(b+1)}$. We
will prove that $S$ is a winning strategy for Maker in this setting.
With this in mind, it suffices to show that there are no active
boxes left at the very end of the game. Suppose not, and let $F$ be
a box which remained active, i.e., $w_M(F) < \alpha |F|$. Clearly
$F$ is not free, since the game has ended. Thus we have $w_M(F) +
w_B(F) = |F|$. Moreover, since Maker played according to $S$, we
must have $\danger(F) \le b(\ln n + 1)$. Hence
\[
  b(\ln n + 1) \ge \danger(F) = w_B(F) - b\cdot w_M(F) = |F| - (b+1) w_M(F)
  > (1 - \alpha (b+1)) |F|.
\]
This implies that $D \le |F| < \frac{b(\ln n + 1)}{1-\alpha(b+1)}$,
which is a contradiction, thereby proving that Maker is the winner,
and concluding the proof of the theorem.

\end{proof}

We turn to prove the variant of P\'osa's lemma for $e$-boosters.

\begin{proof}[Proof of \lemref{lem:boosters}]
Let $D$ be a connected graph for which $|N_{D}(X)\setminus X|\ge
2|X|+2$ holds for every subset $X\subseteq V(D)$ of size $|X|\le
k$. Let $e\in \binom{V(D)}{2}$ be a pair such that the graph
$D\cup\{e\}$ does not contain a Hamilton cycle which uses $e$. We
will prove that the number of $e$-boosters for $D$ is at least
$(k+1)^2/2$.

The idea behind the proof is fairly natural and is based on P\'osa's
\emph{rotation-extension} technique. Let $P=x_0x_1\ldots x_h$ be a
path in $D\cup \{e\}$, starting at a fixed endpoint $x_0$. Suppose $P$
contains $e$, say $e = x_ix_{i+1}$ for some $0 \le i < h$. If $D$
contains an edge $x_jx_h$ for some $0 \le j < h - 1$ such that $j
\ne i$, then one can obtain a new path $P'$ of the same length as
$P$ which contains $e$. The new path is $P'=x_0x_1\ldots
x_jx_hx_{h-1}\ldots x_{j+1}$, obtained by adding the edge $x_jx_h$
and deleting $x_jx_{j+1}$. This operation is called an
\emph{elementary rotation} at $x_j$ with a fixed endpoint $x_0$. We can apply other elementary rotations repeatedly, and if after
a number of rotations, an endpoint $x$ of the obtained path $Q$ is
connected by an edge to a vertex $y$ outside $Q$, then $Q$ can be
extended by adding the edge $xy$.

The power of the rotation-extension technique of P\'osa hinges on
the following fact. Let $P=x_0\ldots x_h$ be a longest path in
$D\cup \{e\}$ containing $e$. Let $\mc P$ be the set of all
paths obtainable from $P$ by a sequence of elementary rotations with
fixed $x_0$. Denote by $R$ the set of the other endpoints (not
$x_0$) of paths in $\mc P$, and by $R^-$ and $R^+$ the sets of
vertices immediately preceding and following the vertices of $R$
along $P$, respectively. We claim that:
\begin{claim}
  \label{eqn:rotation_endpoints}
  $N_D(R)\setminus R \subseteq R^- \cup R^+ \cup e.$
\end{claim}

\begin{proof}[Proof of \clref{eqn:rotation_endpoints}.] Fix $u\in R$,
let $v\in V(D) \setminus (R \cup R^- \cup R^+\cup e)$, and consider
a path $Q\in \mc P$ ending at $u$. If $v\in V(D)\setminus
V(P)$, then $uv\not\in E(D)$, as otherwise the path $Q$ can be
extended by adding $v$, thus contradicting our assumption that $P$
is a longest path in $D\cup \{e\}$ containing $e$.  Suppose now that
$v\in V(P) \setminus (R\cup R^- \cup R^+\cup e)$. Then $v$ has the
same two neighbors in every path in $\mc P$, because an
elementary rotation that removed one of its neighbors along $P$
would, at the same time, put either this neighbor or $v$ itself in
$R$ (in the former case $v\in R^-\cup R^+$). Then if $u$ and $v$ are
adjacent, an elementary rotation at $v$ can be applied to $Q$ (since
$v\not \in e$), and produces a path in $\mc P$ whose endpoint
is a neighbor of $v$ along $P$, a contradiction. Therefore in both
cases $u$ and $v$ are non-adjacent, thereby proving
\clref{eqn:rotation_endpoints}.
\end{proof}

Equipped with \clref{eqn:rotation_endpoints} we turn back to the
proof of the lemma. Again, let $P=x_0x_1\ldots x_h$ be a longest
path in $D\cup \{e\}$ containing $e$, and let $R$, $R^-$, $R^+$ be
as in \clref{eqn:rotation_endpoints}. Note that $|R^-| \le |R|$
and $|R^+| \le |R|-1$, since $x_h\in R$ has no following vertex on
$P$, and thus does not contribute an element to $R^+$. According to
\clref{eqn:rotation_endpoints}, we have
\[
  |N_D(R)\setminus R| \le |R^-\cup R^+\cup e| \le 2|R| + 1,
\]
and it follows that $|R| > k$. We claim that, for each $v\in R$, the
pair $x_0v$ is an $e$-booster for $D$. To prove this claim, fix
$v\in R$, and let $Q\in \mc P$ be a path ending at $v$. Note
that by adding $x_0v$ to $Q$, we turn $Q$ into a cycle $C$
containing $e$. This cycle is either Hamiltonian or $V(Q)\ne V(D)$.
The former case would immediately imply that $x_0v$ is an
$e$-booster for $D$. Thus we may assume that $V(C)=V(Q)\ne V(D)$.
Since $D$ is connected, there exists an edge $yz \in E(D)$
connecting $y\in V(C)$ to $z\not\in V(C)$. We can use the edge $yz$
to obtain a path $P'$ that contains $e$ of length $h + 1$ in the
following way. In $C$ there are two edges incident to $y$, and at
least one of them is not $e$. By removing that edge from $C$ and
adding the edge $yz$, we obtain such path $P'$ of length $h+1$. On
the other hand, because we assumed that $P$ was the longest path in
$D\cup \{e\}$ containing $e$, we must conclude that $x_0v$ is an
$e$-booster for $D$, thereby proving our claim.

To finish the proof of the lemma, fix a subset
$\{y_1,\ldots,y_{k+1}\}$ of $R$. For every $y_i$, there exists a
path $P_i$ ending at $y_i$, that can be obtained from $P$ by a
sequence of elementary rotations. Now fix $y_i$ as the starting
point of $P_i$ and let $Y_i$ be the set of other endpoints of all
paths obtained from $P_i$ by a sequence of elementary rotations with
fixed $y_i$. As before, $|Y_i| \ge k+1$, and all edges connecting
$y_i$ to a vertex in $Y_i$ are $e$-boosters for $D$. Altogether we
have found $(k+1)^2$ pairs $y_iz_{ij}$ for $z_{ij}\in Y_i$. As every
booster is counted at most twice, the conclusion of the lemma
follows.
\end{proof}


\begin{thebibliography}{99}

\bibitem{AloSpe2008}
N.\ Alon and J. H. Spencer, \textbf{The Probabilistic Method}, Wiley,
New-York, 2008.

\bibitem{BCS}
J. Balogh, B. Csaba and W. Samotij, \emph{Local resilience of almost
spanning trees in random graphs}, Random Structures and Algorithms
38, no. 1-2 (2011), 121--139.


\bibitem{Beck}
J.\ Beck, \textbf{Combinatorial Games: Tic-Tac-Toe Theory}, Cambridge
University Press, 2008.


\bibitem{BL} M. Bednarska and T. \L uczak, \emph{Biased positional games for
which random strategies are nearly optimal}, Combinatorica 20 (2000), 477--488.

\bibitem{BKS}
S. Ben-Shimon, M. Krivelevich and B. Sudakov, \emph{Local resilience and
Hamiltonicity Maker-Breaker games in random regular graphs},
Combinatorics, Probability and Computing 20 (2011), 173--211.

\bibitem{BKS2}
S. Ben-Shimon, M. Krivelevich and B. Sudakov, \emph{On the resilience of
Hamiltonicity and optimal packing of Hamilton cycles in random
graphs}, SIAM Journal on Discrete Mathematics 25 (2011), 1176--1193.

\bibitem{BolRand}
B. Bollob{\'a}s, \textbf{Random Graphs}, Cambridge University Press,
2001.



\bibitem{BKT}
J. B\" ottcher, Y. Kohayakawa, and A. Taraz, \emph{Almost spanning
subgraphs of random graphs after adversarial edge removal},
Electronic Notes in Discrete Mathematics 35 (2009), 335--340.


\bibitem{CE}
V. Chv\'atal and P. Erd\H{o}s, \emph{Biased positional games},
Annals of Discrete Mathematics 2 (1978), 221--228.


\bibitem{CFGHL}
D. Clemens, A. Ferber, R. Glebov, D. Hefetz and A. Liebenau,
\emph{Building spanning trees quickly in Maker-Breaker games}, SIAM Journal on Discrete Mathematics, to appear. ArXiv
preprint arXiv:1304.4108 (2013).

\bibitem{FH}
A. Ferber and D. Hefetz, \emph{Weak and strong $k$-connectivity
game}, European Journal of Combinatorics (2014), pp. 169--183.


\bibitem{FHK}
A. Ferber, D. Hefetz and M. Krivelevich, \emph{Fast embedding of spanning
trees in biased Maker-Breaker games}, European Journal of
Combinatorics 33 (2012), 1086-1099.


\bibitem{Frieze}
A. Frieze, \emph{An algorithm for finding hamilton cycles in random
digraphs}, Journal of Algorithms 9 (1988), 181--204.

\bibitem{FK}
A. Frieze and M. Krivelevich, \emph{On two Hamilton cycle problems in
random graphs}, Israel Journal of Mathematics 166 (2008), 221--234.

\bibitem{GS}
H. Gebauer and T. Szab\'o, \emph{Asymptotic random graph intuition
for the biased Connectivity game}, Random Structures and Algorithms,
35 (2009), 431--443.

\bibitem{HKSS2009b}
D. Hefetz, M. Krivelevich, M. Stojakovic and T. Szab\'o, \emph{Fast
winning strategies in Maker-Breaker games}, Journal of Combinatorial
Theory Series B 99 (2009), 39--47.


\bibitem{JLR}
S. Janson, T. \L uczak and  A. Ruci\'nski, \textbf{Random graphs},
Wiley, New York, 2000.

\bibitem{Krive}
M. Krivelevich, \emph{The critical bias for the Hamiltonicity game
is $(1+o(1))n/\ln n$}, Journal of the American Mathematical Society
24 (2011), 125--131.

\bibitem{KLS}
M. Krivelevich, C. Lee and B. Sudakov, \emph{Resilient pancyclicity of
random and pseudo-random graphs}, SIAM Journal on Discrete
Mathematics 24 (2010), 1--16.

\bibitem{LS}
C. Lee and B. Sudakov, \emph{Dirac's theorem for random graphs},
Random Structures and Algorithms 41, no. 3 (2012), 293--305.

\bibitem{Lehman}
A. Lehman, \emph{A solution to the Shannon switching game}, J. Soc.
Indust. Appl. Math. 12 (1964), 687--725.

\bibitem{Posa}
L. P\'osa, \emph{Hamiltonian circuits in random graphs}, Discrete
Math. (1976) 14, 359--364.

\bibitem{SV}
B. Sudakov and V. H. Vu, \emph{Local resilience of graphs}, Random
Structures and Algorithms 33, no. 4 (2008), 409--433.

\bibitem{West}
D.\ B.\ West, \textbf{Introduction to Graph Theory}, Prentice Hall,
2001.

\end{thebibliography}
\end{document}